\documentclass[draft]{birkjour}
\usepackage[noadjust]{cite}
\usepackage{xcolor}
\usepackage{mathrsfs}

\usepackage{amsmath}
\usepackage{amssymb}
\usepackage{amsthm}
\usepackage{extarrows}
\RequirePackage[all]{xy}

\newtheorem{theorem}{Theorem}[section]
\newtheorem{corollary}[theorem]{Corollary}
\newtheorem{lemma}[theorem]{Lemma}

\theoremstyle{definition}
\newtheorem{definition}[theorem]{Definition}
\theoremstyle{remark}
\newtheorem*{remark}{Remark}

\numberwithin{equation}{section}

\newcommand{\BibTeX}{B\kern-0.1emi\kern-0.017emb\kern-0.15em\TeX}
\newcommand{\XYpic}{$\mathrm{X\kern-0.3em\raisebox{-0.18em}{Y}}$-$\mathrm{pic}\,$}

\newcommand{\cl}{C \kern -0.1em \ell}  

\newcommand{\bx}{\boldsymbol{x}}

\newcommand{\be}{\begin{eqnarray*}}
	\newcommand{\ee}{\end{eqnarray*}}

\allowdisplaybreaks

\newcommand{\ed}{\end{document}}
\begin{document}
	\title[Generalized poly-partial-slice monogenic functions]
	{Higher order Teodorescu transform in the theory of generalized partial-slice monogenic functions}
	\author[Manjie Hu]{Manjie Hu}
	\address{%
		School of Mathematical Sciences,\\ Anhui University, Hefei, Anhui, China}
	\email{A23201033@stu.ahu.edu.cn}
	
	\author[Chao Ding]{Chao Ding}
	%
	\address{%
		Center for Pure Mathematics, \\School of Mathematical Sciences,\\ Anhui University, Hefei, Anhui, China}
	\email{cding@ahu.edu.cn}

	\subjclass{30G35, 32A30, 44A05.}
	\keywords{Generalized poly-partial-slice monogenic functions, Generalized poly-partial-slice Cauchy kernel, Teodorescu transform, Fundamental solution, Borel-Pompeiu formula}
	\date{\today}
	\begin{abstract}
		This paper investigates the higher order Teodorescu transform in the context of generalized partial-slice monogenic functions. By introducing the generalized poly-partial-slice Cauchy kernel, we derive some integral formulas including higher order Borel–Pompeiu formula and Cauchy integral formula for generalized poly-partial-slice monogenic functions. Furthermore, we define the higher order Teodorescu transform in this context and establish its existence, continuity, and some mapping properties. These results extend classical integral methods in Clifford analysis to the generalized poly-partial-slice setting.
	\end{abstract}
	\label{page:firstblob}
	\maketitle
	\section{Introduction}\hspace*{\fill} 
	
	Clifford analysis extends complex function theory to higher-dimensional Euclidean spaces. It investigates null solutions of Dirac-type operators defined within real Clifford algebras. Monogenic functions, defined as null solutions to the Dirac operator, represent a fundamental concept. However, the non-commutative nature of Clifford multiplication typically prevents polynomial functions in higher dimensions from being monogenic. Consequently, this inherent limitation restricts the generalization of methods such as power series expansions and the classical Cauchy integral formula. The literature on monogenic functions is extensive, with standard references including \cite{Br1,Co5,Gi1,Gu3}.
	
	To address these limitations, the theory of slice monogenic functions was developed by Gentili and Struppa and later generalized by Colombo, Sabadini, Struppa, Ghiloni, Perotti, and others. These functions are defined by imposing complex analyticity on slices of the underlying space--complex subplanes parameterized by paravector variables, see \cite{Co2,Co6,Ge2,Ge1}. Notably, Colombo, Sabadini, and Struppa \cite{Co2} provided a comprehensive treatment of the functional calculus and operator theory built upon this foundation. Furthermore, the concept of slice regularity has been extended to various settings, including quaternions, octonions, Clifford algebras and real alternative *-algebras, see \cite{A1,A2,Ge3,Di,Gh0}.
	
	\par
	In \cite{Co6}, a significant advance in this direction was Colombo, Gonz\'alez-Cervantes and Sabadini's introduction of the operator $G$, a non-constant coefficient analog of the Cauchy--Riemann operator $\overline{\partial}$, and their establishment of its foundational role in defining slice monogenic functions over paravector domains. Subsequently, this global differential operator for slice regularity has gained some research attention. For example, Ghiloni and Perotti \cite{Gh1} established a framework of global differential equations for slice regular functions over real alternative algebras, refining quaternionic and slice monogenic results and generalizing them. In particular, they introduced the operators $\vartheta$ and $\overline{\vartheta}$, which play a central role in the global differential characterization of slice regularity. Later, Xu and Sabadini \cite{Xu} introduced broader classes of global differential operators with variable coefficients, exemplified by the operator $\overline{\vartheta}$, whose kernel characterizes a class of generalized slice monogenic functions. These operators by Xu and Sabadini have facilitated the development of global integral representations and differential identities paralleling those of classical complex analysis. This collective progress laid the groundwork for extending the theory toward more general differential systems and complex domain geometries. Additionally, the work of Ghiloni and Stoppato \cite{Add1, Add2} on $T$-regular functions provides a powerful unifying framework that encompasses, among others, the theories of monogenic and slice monogenic functions. In reference \cite{A2}, the authors developed the theory of poly-slice monogenic function and proved important Cauchy integral formula. This important progress leads us to think more deeply about poly-slice monogenic function.
	
	\par
	Recent works in \cite{Xu,Xu1,Xu2} have expanded the framework to include generalized partial-slice monogenic functions, defined on domains structured as $\mathbb{R}^p \times \mathbb{R}^q$ and admitting decomposition along partially symmetric variable sets. These functions generalize both classical monogenic and slice monogenic settings. Meanwhile, Ghiloni et al. \cite{Add1,Add2} carried out a deep investigation into the theory of $T$-regular functions, which emcompasses the cases of monogenic functions, slice monogenic functions and generalized partial-slice monogenic functions. In this unified theory, they established fundamental results such as a Cauchy integral formula, series expansions, and an identity principle \cite{Add2}. This establishes a foundational connection between these seemingly distinct theories. This paper can be consided as a continued work to \cite{HD}, which studied integral formulas for generalized partial slice monogenic functions. More specifically, in this paper, we introduce the generalized poly-partial-slice Dirac operators and the Cauchy kernels $\mathcal{E}_{k,\boldsymbol{y}}(\boldsymbol{x})$, defined recursively to reflect the algebraic structure and variable decomposition of the space.
	
	\par
	These kernels $\mathcal{E}_{k,\boldsymbol{y}}(\boldsymbol{x})$ provide the foundation for constructing a sequence of higher order Teodorescu transforms $T_k$. These integral operators act on $L^1$-functions and are intrinsically linked to iterated generalized Dirac-type differential operators. We derive and prove fundamental identities for the kernels and operators, including: a higher order Borel--Pompeiu formula extending classical Clifford analysis results; a generalized Cauchy integral formula for the action of $T_k$ on monogenic-type functions; and a recursive differential relation on fundamental solutions.
	
	This paper is structured as follows. Section $2$ provides fundamental concepts in Clifford algebras and the definition of generalized partial-slice monogenic functions along with some properties. In Section $3$, we define the generalized poly-partial-slice Cauchy kernel and derive its associated Borel-Pompeiu formula and Cauchy integral formula through global differential operators. Section $4$ is devoted to an investigation on the higher order Teodorescu transform corresponding to the previously mentioned Cauchy kernel. Existence and continuity properties of the higher order Teodorescu transform are also studied here.
\section{Preliminaries}
To help explain the problems discussed in this paper and to review relevant background, we first introduce some basic notations and definitions on Clifford algebras and slice functions. For more detailed discussions, see~\cite{Br1,De,Gi1,Co4}.
\subsection{Clifford algebras and monogenic functions}
Let the standard orthonormal basis $\left\{ e_1,e_2,\cdots ,e_n \right\}$ for the $n$-dimensional real Euclidean space $\mathbb{R} ^n$. The real Clifford algebra, denoted by $\mathbb{R} _n$, is constructed from these basis elements with the relations
\begin{align*}  
e_i e_j + e_j e_i = -2\delta_{ij}.  
\end{align*} 
Hence, each element in this algebra can be written as  
\begin{align*}  
\sum\nolimits_A e_A x_A,  
\end{align*}  
where $A = i_1\cdots i_r$ is a multi-index with $i_1 < \cdots < i_r$, $e_A = e_{i_1}e_{i_2}\cdots e_{i_r}$ (with $e_\emptyset = 1$), and $x_A$ is a real number.  
A vector $(x_1, x_2, \dots, x_n)\in \mathbb{R}^n$ can be identified with a $1$-vector in the Clifford algebra through  
\begin{align*}  
(x_1, x_2, \dots, x_n) \mapsto \underline{x} = x_1e_1 + x_2e_2 + \cdots + x_ne_n.  
\end{align*}
Further, we view an element $(x_1, x_2, \dots, x_n)$ in $\mathbb{R}^{n+1}$ as a paravector by writing it as
\begin{align*}
x = x_0 + \underline{x} = x_0 + \sum_{j=1}^n x_j e_j.
\end{align*}
The norm of $x$ is defined by
\begin{align*}
|x|^2 = x_0^2 + x_1^2 + \cdots + x_n^2.
\end{align*}
The real part of $x$, denoted as $x_0$, can also be written as $\text{Re}[x]$. A function $f$ defined on an open set $U \subseteq \mathbb{R}^{n+1}$ with values in $\mathbb{R}_n$ is considered as a function $f(x)$ of the paravector $x$.
\par
\begin{definition}[Monogenic functions] Let $\Omega\subset \mathbb{R}^{n+1}$ be a domain and let $f: \Omega\longrightarrow \mathbb{R}_{n}$ be a function with continuous partial derivatives. The function $f=\sum_{A}e_Af_A$ is called a left monogenic function in $\Omega$ if it satisfies the following generalized Cauchy-Riemann equation
\begin{align*}
\overline{\partial}_xf(x)=(e_0\partial_{x_0}+D_x)f(x)=\sum_{i=0}^{n}e_i\frac{\partial f(x)}{\partial x_i}=0,
\end{align*}
for all $x\in\Omega$, where $D_x=\sum_{i=1}^ne_i\frac{\partial}{\partial x_i}$ is known as the Dirac operator in $\mathbb{R}^n$. Since Clifford multiplication is not commutative in general, there is a similar definition for right monogenic functions.
\end{definition}
\subsection{ Generalized partial-slice monogenic functions}
Now, we introduce some concepts and definitions in the theory of generalized partial-slice monogenic functions, which are needed for later discussions, for more details, see \cite{Xu,Xu1}.
\par
Let $p$ and $q$ be positive integers.
Consider a vector $\boldsymbol{x}\in \mathbb{R} ^{p+q}$, which can be written as
\begin{align*}
	\boldsymbol{x}=\boldsymbol{x}_p+\underline{\boldsymbol{x}}_q\in \mathbb{R} ^p\oplus \mathbb{R} ^q,\quad \boldsymbol{x}_p=\sum_{i=1}^p{x_ie_i},\quad \underline{\boldsymbol{x}}_q=\sum_{i=p+1}^{p+q}{x_ie_i}.
\end{align*}
In the rest of the paper, we always split $\boldsymbol{x}\in \mathbb{R} ^{p+q}$ as $\mathbb{R} ^p\oplus \mathbb{R} ^q$.
\par
In a similar manner, the Dirac operator and the Euler operator in $\mathbb{R}^{p+q}$ can be decomposed as follows.
\begin{align*}
	D_{\boldsymbol{x}} &= \sum_{i=1}^{p+q} e_i \partial_{x_i} = \sum_{i=1}^p e_i \partial_{x_i} + \sum_{i=p+1}^{p+q} e_i \partial_{x_i} = D_{\boldsymbol{x}_p} + D_{\underline{\boldsymbol{x}}_q}, \\
	\mathbb{E}_{\boldsymbol{x}} &= \sum_{i=1}^{p+q} x_i \partial_{x_i} = \sum_{i=1}^p x_i \partial_{x_i} + \sum_{i=p+1}^{p+q} x_i \partial_{x_i} = \mathbb{E}_{\boldsymbol{x}_p} + \mathbb{E}_{\underline{\boldsymbol{x}}_q}.
\end{align*}
We denote
\begin{align*}
	\mathbb{S} = \left\{ \underline{\boldsymbol{x}}_q : {\underline{\boldsymbol{x}}_q}^2 = -1 \right\} = \left\{ \underline{\boldsymbol{x}}_q = \sum_{i=p+1}^{p+q} x_i e_i : \sum_{i=p+1}^{p+q} x_i^2 = 1 \right\},
\end{align*}
which represents the set of unit sphere in $\mathbb{R}^q$.
We also note that any nonzero vector $\underline{\boldsymbol{x}}_q$ can be written uniquely as $\underline{\boldsymbol{x}}_q = r\,\underline{\omega}$, 
where $r = |\underline{\boldsymbol{x}}_q|$ and $ \underline{\omega} = \frac{\underline{\boldsymbol{x}}_q}{|\underline{\boldsymbol{x}}_q|} $ with $ r>0$ and $\underline{\omega}\in\mathbb{S}$. If $\underline{\boldsymbol{x}}_q=0$, we set $r=0$ and note that $\underline{\omega}$ is not uniquely determined because, in that case, any unit vector will work.
\par
We will use the following notation
\begin{align*}
	\Omega_{\underline{\omega}} := \Omega \cap \left( \mathbb{R}^{p} \oplus \underline{\omega} \mathbb{R} \right) \subseteq \mathbb{R}^{p+1},
\end{align*}
where $\Omega$ is an open set in $\mathbb{R}^{p+q}$.
\begin{definition}
	Let $F: D \longrightarrow \mathbb{R}_{p+q} \otimes \mathbb{C}$ be a function on an open set $D \subseteq \mathbb{R}^{p+1}$. We call $F$  a \textbf{stem function} if it is invariant under the reflection of the $(p+1)$-th variable, and its components
	$F_1$ and $F_2$ with $(F = F_1 + iF_2)$ satisfy
	\begin{align*}
		F_1(\boldsymbol{x}_p, -r) = F_1(\boldsymbol{x}_p, r) \quad \text{and} \quad F_2(\boldsymbol{x}_p, -r) = -F_2(\boldsymbol{x}_p, r)
	\end{align*}
	for every point $(\boldsymbol{x}_p, r)$ in $D$.
	\par
	Given a stem function $F$, a (left) generalized partial-slice function can be indcuced as $f = \mathcal{I}(F):\Omega_D\longrightarrow\mathbb{R}_{p+q}$ by
	\begin{align*}
		f(\boldsymbol{x}) = F_1(\boldsymbol{x}') + \underline{\omega}\,F_2(\boldsymbol{x}'),
	\end{align*}
	where we write $\boldsymbol{x}' = (\boldsymbol{x}_p, r)$ and $\boldsymbol{x} = \boldsymbol{x}_p + r\,\underline{\omega}$, with $\underline{\omega} \in \mathbb{S}$.
\end{definition}
We denote the set of all induced generalized partial-slice functions on $\Omega _D$ by
\begin{align*}
	\mathcal{G} \mathcal{S} \left( \Omega _D \right) :=\left\{ f=\mathcal{I} \left( F \right) :  F\,\,is\,\,a\,\,stem\,\,function\,\,on\,\,D \right\}.
\end{align*}
The following representation formula for generalized partial-slice functions plays a significantly important role in the theory of generalized partial-slice monogenic functions.
\begin{theorem}[Representation Formula]\cite{Xu}
	Let $f\in \mathcal{G} \mathcal{S} \left( \Omega _D \right) $. Then it holds that, for every $\boldsymbol{x}=\boldsymbol{x}_p+r\underline{\omega}\in \Omega _D$ with $\underline{\omega }\in \mathbb{S} $,
	\begin{align*}
		f\left( \boldsymbol{x} \right) =&\left( \underline{\omega }-\underline{\omega }_2 \right) \left( \underline{\omega }_1-\underline{\omega }_2 \right) ^{-1}f\left( \boldsymbol{x}_p+r\underline{\omega }_1 \right) \\
		&-\left( \underline{\omega }-\underline{\omega }_1 \right) \left( \underline{\omega }_1-\underline{\omega }_2 \right) ^{-1}f\left( \boldsymbol{x}_p+r\underline{\omega }_2 \right) ,
	\end{align*}
    for all $\underline{\omega }_1\ne \underline{\omega }_2\in \mathbb{S} $. In particular, $\underline{\omega }_1=-\underline{\omega }_2=\underline{\eta }\in \mathbb{S} $,
    \begin{align}\label{Representation Formula}
    	f(\boldsymbol{x})=&\frac{1}{2}\left( 1-\underline{\omega }\underline{\eta } \right) f(\boldsymbol{x}_p+r\underline{\eta })+\frac{1}{2}\left( 1+\underline{\omega }\underline{\eta } \right) f(\boldsymbol{x}_p-r\underline{\eta })\nonumber	\\
    	=&\frac{1}{2}\left( f(\boldsymbol{x}_p+r\underline{\eta })+f(\boldsymbol{x}_p-r\underline{\eta }) \right) +\frac{1}{2}\underline{\omega }\underline{\eta }\left( f(\boldsymbol{x}_p-r\underline{\eta })-f(\boldsymbol{x}_p+r\underline{\eta }) \right).
    \end{align}
\end{theorem}
\begin{definition}
	Let $D$ be a set in $\mathbb{R}^{p+1}$ that is invariant under the reflection $(\bx_p,r)\mapsto (\bx_p,-r)$. The \textbf{$p$-symmetric completion} of $D$, written as $\Omega_D \subset \mathbb{R}^{p+q}$, is defined by  
	\begin{align*}
		\Omega_D = \bigcup_{\underline{\omega} \in \mathbb{S}} \left\{ \bx_p + r \underline{\omega} : \exists\, \bx_p \in \mathbb{R}^p,\, r \geq 0,\, (\bx_p, r) \in D \right\}.
	\end{align*}
\end{definition}

It is useful to note that $\Omega \subset \mathbb{R}^{p+q}$ is $p$-symmetric if and only if there is a corresponding set $D \subset \mathbb{R}^{p+1}$ such that $\Omega = \Omega_D$. In this paper, we always denote a $p$-symmetric set in $\mathbb{R}^{p+q}$ by $\Omega_D$.
\par
The set of all generalized partial-slice functions on $\Omega_D$ is denoted by
$$\mathcal{GS}(\Omega_D) := \Big\{ f = \mathcal{I}(F) : F \text{ is a stem function on } D \Big\}.$$
\begin{definition}
	Let $f \in \mathcal{G}\mathcal{S}(\Omega_D)$. We say that $f$ is \textbf{generalized partial-slice monogenic of type $\boldsymbol{(p,q)}$} if its corresponding stem function $F = F_1 + iF_2$ satisfies the generalized Cauchy-Riemann equations
	\begin{align*}
		\begin{cases}  
			D_{\boldsymbol{x}_p} F_1 - \partial_r F_2 = 0, \\  
			-D_{\boldsymbol{x}_p} F_2 + \partial_r F_1 = 0.
		\end{cases}
	\end{align*}
	Here, $F_1$ and $F_2$ are the real components of $F$. This condition ensures that $f$ follows an extended version of the classical Cauchy-Riemann equations.
\end{definition}
\begin{remark}
The definition of generalized partial-slice monogenic functions is different from the one given in \cite{Xu,Xu1}. Here, we deal with functions defined on the vector space $\mathbb{R}^{p+q}$ instead of the paravector space $\mathbb{R}^{p+q+1}$.
\end{remark}
For any $ C^1 $ function $f: \Omega \longrightarrow \mathbb{R}_{p+q}$,
we define a global differential operator as
\begin{align*}
	\bar{\vartheta} f(\boldsymbol{x}) = D_{\boldsymbol{x}_p} f(\boldsymbol{x}) + \frac{\underline{\boldsymbol{x}}_q}{|\underline{\boldsymbol{x}}_q|^2} \, \mathbb{E}_{\underline{\boldsymbol{x}}_q} f(\boldsymbol{x}).
\end{align*}
Because the term $|\underline{\boldsymbol{x}}_q|^2$ in the denominator can cause singularities, we need to treat this operator with extra care. Hence,  from now on in this paper, we will write $\mathbb{R} _{*}^{p+q}:=\mathbb{R} ^{p+q}\backslash \mathbb{R} ^p$.

\section{Generalized poly-partial-slice integral formulas}

To better study generalized partial-slice monogenic functions in higher order cases, we now introduce the generalized poly-partial-slice Cauchy kernel below.

\begin{definition}[Generalized poly-partial-slice Cauchy kernel]\label{P-Cauchy kernel}
	Let $k\in \mathbb{N}$. Given $\boldsymbol{y}\in \mathbb{R}^{p+q}_{*}$, we define the generalized poly-partial-slice Cauchy kernel $\mathcal{E} _k$ by
	\begin{align*}
		\mathcal{E} _{k,\boldsymbol{y}}\left( \boldsymbol{x} \right) :=\begin{cases}
			\alpha C_{k,p}\cdot \frac{\boldsymbol{y}-\boldsymbol{x}_{\underline{\omega }}}{\sigma _p|\boldsymbol{y}-\boldsymbol{x}_{\underline{\omega }}|^{p-k+1}}+\beta C_{k,p}\cdot \frac{\boldsymbol{y}-\boldsymbol{x}_{-\underline{\omega }}}{\sigma _p|\boldsymbol{y}-\boldsymbol{x}_{-\underline{\omega }}|^{p-k+1}},&k\text{ is odd},\\
			\alpha C_{k,p}\cdot \frac{1}{\sigma _p|\boldsymbol{y}-\boldsymbol{x}_{\underline{\omega }}|^{p-k}}+\beta C_{k,p}\cdot \frac{1}{\sigma _p|\boldsymbol{y}-\boldsymbol{x}_{-\underline{\omega }}|^{p-k}},&k\text{ is even},
		\end{cases}
	\end{align*}
	where $\underline{\omega} = \frac{\underline{\boldsymbol{x}}_q}{|\underline{\boldsymbol{x}}_q|} \in \mathbb{S}$ is the unit vector in the direction of $\underline{\boldsymbol{x}}_q$, and $\sigma_p$ denotes the surface area of the unit sphere in $\mathbb{R}^p$.
\end{definition}
Notice that $\frac{\boldsymbol{y}-\boldsymbol{x}_{\underline{\omega }}}{\sigma _p|\boldsymbol{y}-\boldsymbol{x}_{\underline{\omega }}|^{p-k+1}}$ is actually the poly-monogenic Cauchy kernel, and the generalized poly-partial-slice Cauchy kernel is given by the representation formula (\ref{Representation Formula}) in terms of the poly-monogenic Cauchy kernel.
\par
Next, we establish a recursive identity for the poly-partial-slice Cauchy kernel, which is crucial to construct our integral formulas in the later sections.
\begin{theorem}\label{theta K=K-1}
	Let $\Omega _D\subset \mathbb{R} _{*}^{p+q}$ and $ k \in \mathbb{N} $ with $ k > 1 $. Then for any $ \boldsymbol{x}, \boldsymbol{y} \in \mathbb{R}^{p+q}$ with $\boldsymbol{y}\notin[\bx]$, we have
	\begin{align*}
		\overline{\vartheta }\mathcal{E} _{k,\boldsymbol{y}}\left( \boldsymbol{x} \right) =\mathcal{E} _{k-1,\boldsymbol{y}}\left( \boldsymbol{x} \right).
	\end{align*}
\end{theorem}
\begin{proof}
	Suppose $k = 2n + 1$ for some $n \in \mathbb{N}$. By Definition~\ref{P-Cauchy kernel}, the kernel $\mathcal{E}_{k,\boldsymbol{y}}(\boldsymbol{x})$ for odd $k$ is given by
	\begin{align*}
		\mathcal{E}_{k,\boldsymbol{y}}(\boldsymbol{x}) = \alpha C_{k,p} \cdot \frac{\boldsymbol{y} - \boldsymbol{x}_{\underline{\omega}}}{\sigma_p |\boldsymbol{y} - \boldsymbol{x}_{\underline{\omega}}|^{p - k + 1}} + \beta C_{k,p} \cdot \frac{\boldsymbol{y} - \boldsymbol{x}_{-\underline{\omega}}}{\sigma_p |\boldsymbol{y} - \boldsymbol{x}_{-\underline{\omega}}|^{p - k + 1}},
	\end{align*}
	where $\boldsymbol{x}_{\pm \underline{\omega}} := \boldsymbol{x}_p \pm \underline{\omega} r$. Applying the operator $\overline{\vartheta}$ to $\mathcal{E}_{k,\boldsymbol{y}}(\boldsymbol{x})$, we use the fact that both $\boldsymbol{x}_p$ and $r$ appear linearly in $\boldsymbol{x}_{\pm \underline{\omega}}$, and that the differential operator acts as
	\begin{align*}
		\overline{\vartheta} = D_{\boldsymbol{x}_p} + \frac{\underline{\boldsymbol{x}}_q}{|\underline{\boldsymbol{x}}_q|^2} \, \mathbb{E}_{\underline{\boldsymbol{x}}_q}.
	\end{align*}
	Differentiating the first term
	\begin{align*}
		&\overline{\vartheta }\left( \frac{\boldsymbol{y}-\boldsymbol{x}_{\underline{\omega }}}{|\boldsymbol{y}-\boldsymbol{x}_{\underline{\omega }}|^{p-k+1}} \right) \\
		=& \frac{p\left| \boldsymbol{y}-\boldsymbol{x}\underline{_{\omega}} \right|^{p-k+1}+\left( p-k+1 \right) \left( \boldsymbol{y}-\boldsymbol{x}\underline{_{\omega}} \right) \left| \boldsymbol{y}-\boldsymbol{x}\underline{_{\omega}} \right|^{p-k-1}\left( y_p-x_p \right)}{|\boldsymbol{y}-\boldsymbol{x}_{\underline{\omega }}|^{2\left( p-k+1 \right)}}
		\\
		+&\frac{\left| \boldsymbol{y}-\boldsymbol{x}\underline{_{\omega}} \right|^{p-k+1}+\left( p-k+1 \right) \left( \boldsymbol{y}-\boldsymbol{x}\underline{_{\omega}} \right) \left| \boldsymbol{y}-\boldsymbol{x}\underline{_{\omega}} \right|^{p-k-1}\cdot \frac{\underline{\boldsymbol{x}}_q}{\left| \underline{\boldsymbol{x}}_q \right|}\cdot \left( |\underline{\boldsymbol{y}}_q|-|\underline{\boldsymbol{x}}_q| \right)}{|\boldsymbol{y}-\boldsymbol{x}_{\underline{\omega }}|^{2\left( p-k+1 \right)}} \\
		=&\frac{\left( p+1 \right) \left| \boldsymbol{y}-\boldsymbol{x}\underline{_{\omega}} \right|^{p-k+1}+\left( p-k+1 \right) \left( \boldsymbol{y}-\boldsymbol{x}\underline{_{\omega}} \right) \left( \boldsymbol{y}-\boldsymbol{x}\underline{_{\omega}} \right) \left| \boldsymbol{y}-\boldsymbol{x}\underline{_{\omega}} \right|^{p-k-1}}{|\boldsymbol{y}-\boldsymbol{x}_{\underline{\omega }}|^{2\left( p-k+1 \right)}}\\
		=& k\cdot \frac{1}{|\boldsymbol{y}-\boldsymbol{x}_{\underline{\omega }}|^{p-k+1}}
	\end{align*}
	Therefore,
	\begin{align*}
		\overline{\vartheta} \left( \frac{\boldsymbol{y} - \boldsymbol{x}_{\underline{\omega}}}{|\boldsymbol{y} - \boldsymbol{x}_{\underline{\omega}}|^{p - k + 1}} \right)
		= k\cdot \frac{1}{|\boldsymbol{y}-\boldsymbol{x}_{\underline{\omega }}|^{p-k+1}}.
	\end{align*}
	A similar calculation holds for the second term involving $\boldsymbol{x}_{-\underline{\omega}}$, which leads to
	\begin{align*}
		\overline{\vartheta} \mathcal{E}_{k,\boldsymbol{y}}(\boldsymbol{x}) = \mathcal{E}_{k-1, \boldsymbol{y}}(\boldsymbol{x}),
	\end{align*}
	as defined in Definition~\ref{P-Cauchy kernel} for even \(k - 1\).
	\par
	A similar argument applies when $k = 2n$ is even. In this case, $\mathcal{E}_k$ involves reflection-symmetric terms of the form $\boldsymbol{y} - \boldsymbol{x}_{\pm \underline{\omega}}$. Applying $\overline{\vartheta}$ to each term again reduces the homogeneity of the denominator by two, and the result matches the structure of $\mathcal{E}_{k-1}$.	
	By the same reasoning, one can verify that
	\begin{align*}
		\overline{\vartheta} \mathcal{E}_{k-1} = \mathcal{E}_{k-2},
	\end{align*}
	regardless of the parity of $k-1$. Thus, the identity holds for all $k > 1$.
\end{proof}
With Theorem \ref{theta K=K-1} above, it is obvious that we can reach the following conclusion.
\begin{theorem}\label{kernel=0}
	For  each $k\in \mathbb{N} $, and $ \boldsymbol{x}, \boldsymbol{y} \in \mathbb{R}^{p+q}$ with $\boldsymbol{y}\notin[\bx]$, the generalized poly-partial-slice Cauchy kernel $ \mathcal{E}_{k,\boldsymbol{y}}(\boldsymbol{x}) $ satisfies 
	\begin{align*}
		\overline{\vartheta }^k\mathcal{E}_{k,\boldsymbol{y}}(\boldsymbol{x}) =0.
	\end{align*}	
\end{theorem}
\begin{proof}
	From Theorem \ref{theta K=K-1}, we can easily obtain
	\begin{align*}
		\overline{\vartheta }^k\mathcal{E}_{k,\boldsymbol{y}}(\boldsymbol{x})  =\overline{\vartheta }^{k-1}\mathcal{E}_{k-1,\boldsymbol{y}}(\boldsymbol{x})  =\cdots \cdots =\overline{\vartheta }\mathcal{E}_{1,\boldsymbol{y}}(\boldsymbol{x}) =0.
	\end{align*}
	The proof is complete.
\end{proof}
It is easy to observe that the operator $\overline{\vartheta}$ becomes the Dirac operator in $\mathbb{R}^{p+1}$ when it is restricted from $\mathbb{R}_{*}^{p+q}$ to $\mathbb{R}^{p+1}$. Hence, we immediately obtain the following Stokes' Theorem from the classical result of the Dirac operator given in \cite[Theorem A.2.22]{Gu3}.
\begin{theorem}[Stokes Theorem]\label{Stokes Theorem Slice}\cite[Theorem A.2.22]{Gu3}
	If $U$ is a domain in $\mathbb{R}_{*}^{p+q}$ such that $U_{\underline{\eta}} \subset \Omega_{\underline{\eta}}$ is a bounded domain in $\mathbb{R}^{p+1}$ with smooth boundary $\partial U_{\underline{\eta}} \subset \Omega_{\underline{\eta}}$ for some $\underline{\eta} \in \mathbb{S}$, and let $f, g \in C^1(\overline{U_{\underline{\eta}}})$. Then we have
	\begin{align}\label{S-T}
		\int_{\partial U_{\underline{\eta}}} f \, dS_{\underline{\eta}} \, g = \int_{U_{\underline{\eta}}} \left[ (f \overline{\vartheta}) g + f (\overline{\vartheta} g) \right] \, dV_{\underline{\eta}},
	\end{align}
	where $dS_{\underline{\eta}}$ is the oriented surface element on $\partial U_{\underline{\eta}}$, and $dV_{\underline{\eta}}$ is the volume element in $U_{\underline{\eta}}$.
\end{theorem}

We can now state a Borel-Pompeiu formula in terms of the generalized poly-partial-slice Cauchy kernel as the following.
\begin{theorem}[Borel-Pompeiu formula]\label{Borel-Pompeiu}
	 Let the function $f\in C^k(\Omega _D)$. If $U$ is a domain in $\mathbb{R} _{*}^{p+q}$ such that $U_{\underline{\eta}}\subset \Omega _{\underline{\eta}}$ is a bounded domain in $\mathbb{R}^{p+1}$ with smooth boundary $\partial U_{\underline{\eta}}\subset \Omega _{\underline{\eta}}$ for some $\underline{\eta}\in \mathbb{S}$, then for any $\boldsymbol{x}\in U$, we have
	\begin{align}\label{B-P F}
		f(\boldsymbol{x})&=(-1)^k\int_{U _{\underline{\eta }}}{\mathcal{E}_{k,\boldsymbol{y}}(\boldsymbol{x})\overline{\vartheta }^kf(\boldsymbol{y})dV_{\underline{\eta }}(\boldsymbol{y})}\nonumber\\
		&-\sum_{j=1}^k{\int_{\partial U _{\underline{\eta }}}{(-1)^j\mathcal{E}_{j,\boldsymbol{y}}(\boldsymbol{x})dS_{\underline{\eta }}\overline{\vartheta }^{j-1}f(\boldsymbol{y})}},
	\end{align}
    where $dV_{\underline{\eta}}$ is the volume element on $U_{\underline{\eta}}$ on $\partial U_{\underline{\eta}}$, $dS_{\underline{\eta}}=n(\boldsymbol{y})d\sigma_{\underline{\eta}}(\boldsymbol{y})$, and $n(\boldsymbol{y})$ is the unit outward normal vector of $\partial U_{\underline{\eta}}$, $d\sigma_{\underline{\eta}}(\boldsymbol{y})$ is the surface area element on $\partial U_{\underline{\eta}}$.
\end{theorem}
\begin{proof}
	We prove the result by mathematical induction on $ k $.
	
	When $ k = 1 $, we apply the Cauchy-Pompeiu formula (see \cite[Theorem 4.3]{Xu}), which gives
	\begin{align}\label{C-PF(j=1)}
		-\int_{\partial U _{\underline{\eta }}}{\mathcal{E} _1(\boldsymbol{x},\boldsymbol{y})dS_{\underline{\eta }}f(\boldsymbol{y})}=-f(\boldsymbol{x})-\int_{U _{\underline{\eta }}}{\mathcal{E} _1(\boldsymbol{x},\boldsymbol{y})\overline{\vartheta }f(\boldsymbol{y})dV_{\underline{\eta }}(\boldsymbol{y})}.
	\end{align}
	Assume the formula holds for $ k = n $. We prove it for $ k = n+1 $.
	
	By Stokes' Theorem given in equation (\ref{S-T}), we have
	\begin{align*}
		&\int_{\partial U_{\underline{\eta}}} \mathcal{E}_{n+1}(\boldsymbol{x}, \boldsymbol{y}) \, dS_{\underline{\eta}} \, \overline{\vartheta}^n f(\boldsymbol{y}) \\
		=& \int_{U_{\underline{\eta}}} \left[ (\overline{\vartheta} \mathcal{E}_{n+1})(\boldsymbol{x}, \boldsymbol{y}) \cdot \overline{\vartheta}^n f(\boldsymbol{y}) + \mathcal{E}_{n+1}(\boldsymbol{x}, \boldsymbol{y}) \cdot \overline{\vartheta}^{n+1} f(\boldsymbol{y}) \right] dV_{\underline{\eta}}(\boldsymbol{y}).
	\end{align*}
	
	Using the identity $ \overline{\vartheta} \mathcal{E}_{n+1} = \mathcal{E}_n $, this becomes
	\begin{align*}
		&\int_{\partial U_{\underline{\eta}}} \mathcal{E}_{n+1}(\boldsymbol{x}, \boldsymbol{y}) \, dS_{\underline{\eta}} \, \overline{\vartheta}^n f(\boldsymbol{y})\\
		=& \int_{U_{\underline{\eta}}} \mathcal{E}_n(\boldsymbol{x}, \boldsymbol{y}) \, \overline{\vartheta}^n f(\boldsymbol{y}) \, dV_{\underline{\eta}}(\boldsymbol{y}) 
		+ \int_{U_{\underline{\eta}}} \mathcal{E}_{n+1}(\boldsymbol{x}, \boldsymbol{y}) \, \overline{\vartheta}^{n+1} f(\boldsymbol{y}) \, dV_{\underline{\eta}}(\boldsymbol{y}).
	\end{align*}
	
	Adding this to the induction hypothesis yields
	\begin{align*}
		f(\boldsymbol{x}) =& (-1)^n \int_{U_{\underline{\eta}}} \mathcal{E}_n(\boldsymbol{x}, \boldsymbol{y}) \, \overline{\vartheta}^n f(\boldsymbol{y}) \, dV_{\underline{\eta}}(\boldsymbol{y}) \\
		-& \sum_{j=1}^n \int_{\partial U_{\underline{\eta}}} (-1)^j \mathcal{E}_j(\boldsymbol{x}, \boldsymbol{y}) \, dS_{\underline{\eta}} \, \overline{\vartheta}^{j-1} f(\boldsymbol{y}) \\
		=& (-1)^{n+1} \int_{U_{\underline{\eta}}} \mathcal{E}_{n+1}(\boldsymbol{x}, \boldsymbol{y}) \, \overline{\vartheta}^{n+1} f(\boldsymbol{y}) \, dV_{\underline{\eta}}(\boldsymbol{y})  \\
		-& \sum_{j=1}^{n+1} \int_{\partial U_{\underline{\eta}}} (-1)^j \mathcal{E}_j(\boldsymbol{x}, \boldsymbol{y}) \, dS_{\underline{\eta}} \, \overline{\vartheta}^{j-1} f(\boldsymbol{y}).
	\end{align*}
	Hence, the formula holds for $ k = n+1 $, and the proof is complete.
\end{proof}
An immediate consequence of the Borel-Pompeiu formula introduced above is a Cauchy integral formula stated below.
\begin{theorem}[Cauchy integral formula]\label{Cauchy integral formula}
	Let $f\in C^k\left( \Omega _D \right) \cap \mathrm{ker}\overline{\vartheta }^k$. If $U$ is a domain in $\mathbb{R} _{*}^{p+q}$ such that $U_{\underline{\eta}}\subset \Omega _{\underline{\eta}}$ is a bounded domain in $\mathbb{R}^{p+1}$ with smooth boundary $\partial U_{\underline{\eta}}\subset \Omega _{\underline{\eta}}$ for some $\underline{\eta}\in \mathbb{S}$, then 
	\begin{align*}
		\sum_{j=0}^{k-1}{(-1)^j\int_{\partial U_{\underline{\eta }}}{\mathcal{E}_{j+1,\boldsymbol{y}}(\boldsymbol{x})dS_{\underline{\eta }}\overline{\vartheta }^jf(\boldsymbol{y})}}=\begin{cases}
			0,\quad &\boldsymbol{x}\notin U,\\
			f\left( \boldsymbol{x} \right) ,&\boldsymbol{x}\in U.\\
		\end{cases}
	\end{align*}
	
\end{theorem}
\begin{proof}
	If the function $f\in \ker\overline{\vartheta }^k$, meaning it satisfies $\overline{\vartheta }^kf=0$, then according to Theorem \ref{Borel-Pompeiu}, the volume integral term in the Borel-Pompeiu formula vanishes. That is,
	\begin{align*}
		(-1)^k\int_{U_{\underline{\eta }}}{\mathcal{E}_{k,\boldsymbol{y}}(\boldsymbol{x})\overline{\vartheta }^kf(\boldsymbol{y})dV_{\underline{\eta }}(\boldsymbol{y})}=0.
	\end{align*}
	Thus, equation (\ref{B-P F}) can be simplified to
	\begin{align*}
		f(x)=-\sum_{j=1}^k{\int_{\partial U_{\underline{\eta }}}{(-1)^j\mathcal{E}_{j,\boldsymbol{y}}(\boldsymbol{x})dS_{\underline{\eta }}\overline{\vartheta }^{j-1}f(\boldsymbol{y})}}.
	\end{align*}
	When $\boldsymbol{x}\notin U$, it is outside the integration region. By Theorem \ref{kernel=0}, the integral on the right-hand side is zero. Therefore,
	\begin{align*}
		\sum_{j=0}^{k-1}{(-1)^j\int_{\partial U_{\underline{\eta }}}{\mathcal{E}_{j+1,\boldsymbol{y}}(\boldsymbol{x})dS_{\underline{\eta }}\overline{\vartheta }^jf(\boldsymbol{y})}}=0.
	\end{align*} 
	When $ \boldsymbol{x} \in U $, it lies within the integration domain. We can prove the following result using the classical approach of integral operators
	\begin{align*}
		\sum_{j=0}^{k-1}{(-1)^j\int_{\partial U_{\underline{\eta }}}{\mathcal{E}_{j+1,\boldsymbol{y}}(\boldsymbol{x})dS_{\underline{\eta }}\overline{\vartheta }^jf(\boldsymbol{y})}}=f\left( \boldsymbol{x} \right).
	\end{align*}
	Combining the above two cases, we conclude that
	\begin{align*}
		\sum_{j=0}^{k-1}{(-1)^j\int_{\partial U_{\underline{\eta }}}{\mathcal{E}_{j+1,\boldsymbol{y}}(\boldsymbol{x})dS_{\underline{\eta }}\overline{\vartheta }^jf(\boldsymbol{y})}}=\begin{cases}
			0,\quad &\boldsymbol{x}\notin U,\\
			f\left( \boldsymbol{x} \right) ,&\boldsymbol{x}\in U,\\
		\end{cases}
	\end{align*}
	which completes the proof.
\end{proof}

\section{Existence and continuity of the higher order Teodorescu transform}
Now, we introduce the higher order Teodorescu transform $T_k$ as
\begin{align*}
	T_kf(\boldsymbol{x}):=\int_{\Omega _D}{\mathcal{E}_{k,\boldsymbol{y}}(\boldsymbol{x}) f\left( \boldsymbol{y} \right) dV\left( \boldsymbol{y} \right)},
\end{align*}
where $ \mathcal{E}_{k,\boldsymbol{y}}(\boldsymbol{x}) $ denotes the generalized poly-partial-slice Cauchy kernel as specified in Definition \ref{P-Cauchy kernel}, and $dV(\boldsymbol{y})$ denotes the volume element in $\Omega_D \subset \mathbb{R}^{p+q}$. 

In the following content, we will prove the existence and continuity of $T_k$. We also refer to \cite{Gu4,Gu2,Gu1} for more discussions on the Teodorescu transform in the quaternionic setting, where similar conclusions were obtained. \par

Before proving the main results, we first give a useful lemma.  
It helps us estimate integrals with weak singularities.
\begin{lemma}\label{lemma}
	Let $\Omega_D$ be a bounded domain and $f\in L^1\left( \Omega _D\right) $. Then for any $\gamma>0 $, the integral 
	\begin{align*}
		F_{\pm}\left( \boldsymbol{x} \right) =\int_{\Omega _D}{\frac{1}{|\boldsymbol{y}-\boldsymbol{x}_{\pm \underline{\omega }}|^{p-\gamma}}f\left( \boldsymbol{y} \right) dV(\boldsymbol{y})}
	\end{align*}
	belongs to $L^1\left( \Omega _D \right) $.
\end{lemma}
\begin{proof}
	Notice that $\Omega _D$ is bounded and $\gamma>0$. To show the result, we first look at the double integral 
	\begin{align*}
		\int_{\Omega _D}{\left( \int_{\Omega _D}{\frac{1}{|\boldsymbol{y}-\boldsymbol{x}_{\underline{\omega}}|^{p-\gamma}}dV(\boldsymbol{x})} \right) |f\left( \boldsymbol{y} \right) |dV(\boldsymbol{y})}.
	\end{align*}
	According to a result given by Hedberg (\cite[Theorem~1]{Hed72}), there is a constant $C$ such that
	\begin{align*}
		\int_{\Omega _D}{\frac{1}{|\boldsymbol{y}-\boldsymbol{x}_{\underline{\omega }}|^{p-\gamma}}dV(\boldsymbol{y})}\leqslant C\left( \text{diam}B \right) ^{\gamma},
	\end{align*}
	where $B$ is the smallest cube that contains $\Omega_D$ and $\text{diam}B$ stands for the diameter of $B$. This means that the inner integral is bounded. So we can write
	\begin{align*}
		&\int_{\Omega _D}{\left( \int_{\Omega _D}{\frac{1}{|\boldsymbol{y}-\boldsymbol{x}_{\underline{\omega }}|^{p-\gamma}}dV(\boldsymbol{x})} \right) |f\left( \boldsymbol{y} \right) |dV(\boldsymbol{y})}\\
		\leqslant &M_{\left( \Omega _D,\gamma \right)} \int_{\Omega _D}{|f\left( \boldsymbol{y} \right) |dV(\boldsymbol{y})}
		=M_{\left( \Omega _D,\gamma \right)} \left\| f \right\| _{L^1},
	\end{align*}
	where $M_{\left( \Omega _D,\gamma \right)}$ is a constant depending on $\Omega _D$ and $\gamma$. The right-hand side is just $M_{\left( \Omega _D,\gamma \right)} \left\| f \right\| _{L^1}$, which is finite since $f\in L^1\left( \Omega _D \right)$.
	Now, by Fubini’s Theorem, we get
	\begin{align*}
		&\int_{\Omega _D}{\left( \int_{\Omega _D}{\frac{1}{|\boldsymbol{y}-\boldsymbol{x}_{\underline{\omega }}|^{p-\gamma}}dV(\boldsymbol{x})} \right) |f\left( \boldsymbol{y} \right) |dV(\boldsymbol{y})}\\
		=&\int_{\Omega _D}{\left( \int_{\Omega _D}{\frac{1}{|\boldsymbol{y}-\boldsymbol{x}_{\underline{\omega }}|^{p-\gamma}}|f\left( \boldsymbol{y} \right) |dV(\boldsymbol{y})} \right) dV(\boldsymbol{x})}\\
		=&\int_{\Omega _D}{F\left( \boldsymbol{x} \right) dV(\boldsymbol{x})}.
	\end{align*}
	Since this integral is also finite, it means that $F_+(\boldsymbol{x})$ is in $L^1\left( \Omega _D \right)$.
	Similarly, we can obtain that $F_-(\boldsymbol{x})$ is in $L^1\left( \Omega _D \right)$. The conclusion is proved.
\end{proof}
With this lemma, we can now show that the singular integral $T_k f$ is well-defined.
\begin{theorem}\label{exist theorem}
	For $k \in \mathbb{N}$ and $f\in L^1\left( \Omega _D \right) $, the integral $T_kf(\boldsymbol{x})$ exists for all $\boldsymbol{x}\in \Omega_D$.
\end{theorem}
\begin{proof}
	We consider two cases, depending on whether $ k $ is odd or even.
	
	\textbf{Case 1: $k$ is odd.}
	
	First, we estimate the norm of $ \mathcal{E}_{k,\boldsymbol{y}}(\boldsymbol{x}) $. When $ k \leq p $, we have
	\begin{align}\label{estimate k odd}
		\left| \mathcal{E}_{k,\boldsymbol{y}}(\boldsymbol{x}) \right|&=\left| \alpha C_{k,p}\cdot \frac{\boldsymbol{y}-\boldsymbol{x}_{\underline{\omega }}}{\sigma _p|\boldsymbol{y}-\boldsymbol{x}_{\underline{\omega }}|^{p-k+1}}+\beta C_{k,p}\cdot \frac{\boldsymbol{y}-\boldsymbol{x}_{-\underline{\omega }}}{\sigma _p|\boldsymbol{y}-\boldsymbol{x}_{-\underline{\omega }}|^{p-k+1}} \right|\nonumber\\
		&\leqslant C_1\cdot \frac{1}{|\boldsymbol{y}-\boldsymbol{x}_{\underline{\omega }}|^{p-k}}+C_2\cdot \frac{1}{|\boldsymbol{y}-\boldsymbol{x}_{-\underline{\omega }}|^{p-k}},
	\end{align}
	where $C_1$ and $C_2$ are constants.
	Using Lemma \ref{lemma} leads to existence of $T_kf(\boldsymbol{x})$ for almost all $\boldsymbol{x} \in \Omega _D$.
	
	If $k>p$ is fixed, we have
	\begin{align*}
		\left| \mathcal{E}_{k,\boldsymbol{y}}(\boldsymbol{x}) \right|\leqslant C_1\cdot |\boldsymbol{y}-\boldsymbol{x}_{\underline{\omega }}|^{k-p}+C_2\cdot |\boldsymbol{y}-\boldsymbol{x}_{-\underline{\omega }}|^{k-p}.
	\end{align*}

	Looking at the right-hand side of this inequality, we can clearly see that when $ k > p $, the kernel becomes bounded near $ \boldsymbol{x} = \boldsymbol{y} $. Thus, $ T_k f(\boldsymbol{x}) $ is well-defined.
	
	\textbf{Case 2: $ k $ is even.}
	
     When $ k \leq p $, we get
	\begin{align}\label{estimate k even}
		\left| \mathcal{E}_{k,\boldsymbol{y}}(\boldsymbol{x}) \right|&=\left| \alpha C_{k,p}\cdot \frac{1}{\sigma _p|\boldsymbol{y}-\boldsymbol{x}_{\underline{\omega }}|^{p-k}}+\beta C_{k,p}\cdot \frac{1}{\sigma _p|\boldsymbol{y}-\boldsymbol{x}_{-\underline{\omega }}|^{p-k}} \right|\nonumber\\
		&\leqslant C_3\cdot \frac{1}{|\boldsymbol{y}-\boldsymbol{x}_{\underline{\omega }}|^{p-k}}+C_4\cdot \frac{1}{|\boldsymbol{y}-\boldsymbol{x}_{-\underline{\omega }}|^{p-k}},
	\end{align}
	where $C_3$ and $C_4$ are constants.
	This allows us to apply Lemma \ref{lemma}, just like in the odd case.
	
	For $ k > p $, we again find that
	\begin{align*}
		\left| \mathcal{E}_{k,\boldsymbol{y}}(\boldsymbol{x}) \right|\leqslant C_3\cdot |\boldsymbol{y}-\boldsymbol{x}_{\underline{\omega }}|^{k-p}+C_4\cdot |\boldsymbol{y}-\boldsymbol{x}_{-\underline{\omega }}|^{k-p}.
	\end{align*}
	This implies boundedness near $ \boldsymbol{x} = \boldsymbol{y} $, so $ T_k f(\boldsymbol{x}) $ is well-defined. Hence, we can conclude that $ T_k f(\boldsymbol{x}) $ exists for all $ \boldsymbol{x} \in \Omega_D $ in every case.
\end{proof}
After knowing that $T_k f$ exists, we also want to know if it keeps the integrability of $f$.  
The next theorem shows that $T_k$ maps $L^1$-functions into $L^s$-functions under some conditions on $s$.
\begin{theorem}
	Let $\Omega_D$ be a bounded domain and $f\in L^1\left( \Omega _D\right) $. If
	\begin{itemize}
		\item[\rm(i)] $1 \leq s < \frac{p}{p-k+1}$, with $k$ is odd and $p >k-1$,
		\item[\rm(ii)] $1 \leq s < \frac{p}{p-k}$, with $k$ is even and $p > k$,
		\item[\rm(iii)] $1 \leq s \leq +\infty$, with $k$ is odd and $p \le k-1$, or with $k$ is even and $p \le k$,
	\end{itemize}
	then $T_k f \in L^s(\Omega_D)$, and there exists a constant $M_k>0$ such that  
	\begin{align*}
		\|T_k f\|_{L^s} \leq M_k \|f\|_{L^1}.
	\end{align*}
\end{theorem}

\begin{proof}
	Using the pointwise estimates for $\mathcal{E}_{k,\boldsymbol{y}}$ established in \eqref{estimate k odd} and \eqref{estimate k even}, we have
	\begin{align*}
		|\mathcal{E}_{k,\boldsymbol{y}}(\boldsymbol{x})| \leq C \left( \frac{1}{|\boldsymbol{y}-\boldsymbol{x}_{\underline{\omega}}|^\lambda} + \frac{1}{|\boldsymbol{y}-\boldsymbol{x}_{-\underline{\omega}}|^\lambda} \right),
	\end{align*}
	where $\lambda = p - k + 1$ if $k$ is odd, and $\lambda = p - k$ if $k$ is even, we obtain
	\begin{align*}
		|T_k f(\boldsymbol{x})| \leq C \left( \int_{\Omega_D} \frac{|f(\boldsymbol{y})|}{|\boldsymbol{y}-\boldsymbol{x}_{\underline{\omega}}|^\lambda}\,dV(\boldsymbol{y}) + \int_{\Omega_D} \frac{|f(\boldsymbol{y})|}{|\boldsymbol{y}-\boldsymbol{x}_{-\underline{\omega}}|^\lambda}\,dV(\boldsymbol{y}) \right).
	\end{align*}
	These integrals are finite provided $\lambda < p$, hence $T_k f(\boldsymbol{x})$ is well-defined.
	
	We now estimate $\|T_k f\|_{L^s}$ by considering each case separately.
	
	\noindent\textbf{Case (i):} $k$ is odd, $p > k-1$, and $1 \leq s < \frac{p}{p-k+1}$.
	
	Here $\lambda = p - k + 1$. Let $1 \le s < +\infty$ with $\frac{1}{s} + \frac{1}{t} = 1$. By Minkowski's inequality and Fubini's theorem,
	\begin{align*}
		\|T_k f\|_{L^s} &= \left( \int_{\Omega_D} \left| \int_{\Omega_D} \mathcal{E}_{k,\boldsymbol{y}}(\boldsymbol{x}) f(\boldsymbol{y})\,dV(\boldsymbol{y}) \right|^s dV(\boldsymbol{x}) \right)^{1/s} \\
		&\leq \int_{\Omega_D} \left( \int_{\Omega_D} |\mathcal{E}_{k,\boldsymbol{y}}(\boldsymbol{x})|^s\, dV(\boldsymbol{x}) \right)^{1/s} |f(\boldsymbol{y})|\,dV(\boldsymbol{y}).
	\end{align*}
	Note that
	\begin{align*}
		|\mathcal{E}_{k,\boldsymbol{y}}(\boldsymbol{x})|^s \leq 2^{s-1} C^s \left( \frac{1}{|\boldsymbol{y}-\boldsymbol{x}_{\underline{\omega}}|^{\lambda s}} + \frac{1}{|\boldsymbol{y}-\boldsymbol{x}_{-\underline{\omega}}|^{\lambda s}} \right).
	\end{align*}
	
	To ensure the $L^s$ boundedness, we need the function
	\begin{align*}
		\boldsymbol{x} \mapsto |\mathcal{E}_{k,\boldsymbol{y}}(\boldsymbol{x})|^s
	\end{align*}
	to be integrable uniformly in $\boldsymbol{y}$. Consider the integral
	\begin{align*}
		\int_{\Omega_D} \frac{1}{|\boldsymbol{y}-\boldsymbol{x}_{\underline{\omega}}|^{\lambda s}} dV(\boldsymbol{x}).
	\end{align*}
	By the change of variable $\boldsymbol{z} = \boldsymbol{x}_{\underline{\omega}}$, the integral becomes
	\begin{align*}
		\int_{\Omega_D} \frac{1}{|\boldsymbol{y}-\boldsymbol{x}_{\underline{\omega}}|^{\lambda s}} dV(\boldsymbol{x}) = \int_{\widetilde{\Omega}_D} \frac{1}{|\boldsymbol{y}-\boldsymbol{z}|^{\lambda s}} J(\boldsymbol{z}) dV(\boldsymbol{z}),
	\end{align*}
	where $J(\boldsymbol{z})$ is the Jacobian of the transformation and $\widetilde{\Omega}_D$ is the image of $\Omega_D$ under this mapping. The dominant singularity occurs when $\boldsymbol{z} \to \boldsymbol{y}$, where the integrand behaves similarly to $|\boldsymbol{y}-\boldsymbol{z}|^{-\lambda s}$. 
	
	To analyze this singularity, consider a small ball $B(\boldsymbol{y}, \varepsilon) \subset \mathbb{R}^p$ centered at $\boldsymbol{y}$ with radius $\varepsilon > 0$. In spherical coordinates centered at $\boldsymbol{y}$, we have
	\begin{align*}
		\int_{B(\boldsymbol{y}, \varepsilon)} \frac{1}{|\boldsymbol{y}-\boldsymbol{z}|^{\lambda s}} dV(\boldsymbol{z}) = \sigma_p \int_0^\varepsilon \frac{1}{r^{\lambda s}} \cdot r^{p-1} dr = \sigma_p \int_0^\varepsilon r^{p-1-\lambda s} dr,
	\end{align*}
	where $\sigma_p$ is the surface area of the unit sphere in $\mathbb{R}^p$.
	
	This integral converges if and only if
	\begin{align*}
		\int_0^\varepsilon r^{p-1-\lambda s} dr < \infty \quad \Longleftrightarrow \quad p-1-\lambda s > -1 \quad \Longleftrightarrow \quad \lambda s < p.
	\end{align*}
	
	The condition $\lambda s < p$ becomes
	\begin{align*}
		(p - k + 1)s < p \quad \Longleftrightarrow \quad s < \frac{p}{p - k + 1}.
	\end{align*}
	Under this condition, we have
	\begin{align*}
		\sup_{\boldsymbol{y}\in \overline{\Omega_D}} \left( \int_{\Omega_D} |\mathcal{E}_{k,\boldsymbol{y}}(\boldsymbol{x})|^s\,dV(\boldsymbol{x}) \right)^{1/s} \leq M_{k_1} < +\infty,
	\end{align*}
	and therefore
	\begin{align*}
		\|T_k f\|_{L^s} \leq M_{k_1} \|f\|_{L^1}.
	\end{align*}
	
	\noindent\textbf{Case (ii):} $k$ is even, $p > k$, and $1 \leq s < \frac{p}{p-k}$.
	
	Here $\lambda = p - k$. The proof is analogous to Case (i). The integrability condition becomes
	\begin{align*}
		(p - k)s < p \quad \Longleftrightarrow \quad s < \frac{p}{p - k}.
	\end{align*}
	Thus, there exists a constant $M_{k_2} > 0$ such that
	\begin{align*}
		\|T_k f\|_{L^s} \leq M_{k_2} \|f\|_{L^1}.
	\end{align*}
	
	\noindent\textbf{Case (iii):} We consider two subcases:
	\begin{itemize}
		\item[(a)] $k$ is odd, $p \le k-1$, and $1 \leq s \leq +\infty$.
		\item[(b)] $k$ is even, $p \le k$, and $1 \leq s \leq +\infty$.
	\end{itemize}
	
	For (a), we have $\lambda = p - k + 1 \le 0$, so the kernel $\mathcal{E}_{k,\boldsymbol{y}}(\boldsymbol{x})$ is bounded. That is, there exists a constant $C_k > 0$ such that
	\begin{align*}
		|\mathcal{E}_{k,\boldsymbol{y}}(\boldsymbol{x})| \le C_k.
	\end{align*}
	It follows that
	\begin{align*}
		|T_k f(\boldsymbol{x})| \le C_k \int_{\Omega_D} |f(\boldsymbol{y})|\,dV(\boldsymbol{y}) = C_k \|f\|_{L^1},
	\end{align*}
	so $\|T_k f\|_{L^\infty} \le C_k \|f\|_{L^1}$. For $1 \le s < \infty$, the $L^\infty$ bound and the boundedness of $\Omega_D$ imply that $T_k f \in L^s(\Omega_D)$ and there exists a constant $M_{k_3} > 0$ such that
	\begin{align*}
		\|T_k f\|_{L^s} \le M_{k_3} \|f\|_{L^1}.
	\end{align*}
	
	For (b), $\lambda = p - k \le 0$, so the kernel $\mathcal{E}_{k,\boldsymbol{y}}(\boldsymbol{x})$ is also bounded. The proof is identical to (a), and there exists a constant $M_{k_4} > 0$ such that
	\begin{align*}
		\|T_k f\|_{L^s} \le M_{k_4} \|f\|_{L^1}.
	\end{align*}
	
	Combining all cases, taking $M_k = \max\{ M_{k_1}, M_{k_2}, M_{k_3}, M_{k_4} \}$ yields the desired result.
\end{proof}
 Next, we investigate the behavior of the differential operators $ \overline{\vartheta}^k $ and the integral operators $ T_k $ under the group $ \mathrm{GRAV}(\mathbb{R}^p \oplus \mathbb{R}^q) $, which preserves the geometry of $ p $-symmetric slice domains.

We begin by recalling that the group $\mathrm{GRAV}(\mathbb{R}^p \oplus \mathbb{R}^q)$, as defined in \cite{Xu}, has the following structure
$$
\mathrm{GRAV}\left( \mathbb{R} ^p\oplus \mathbb{R} ^q \right) =\left\langle 
\left( \begin{matrix}
	1&		b\\
	0&		1\\
\end{matrix} \right) ,
\left( \begin{matrix}
	a&		0\\
	0&		a^{-1}\\
\end{matrix} \right) ,
\left( \begin{matrix}
	0&		1\\
	-1&		0\\
\end{matrix} \right) ,
\left( \begin{matrix}
	\lambda&		0\\
	0&		\lambda ^{-1}\\
\end{matrix} \right) 
\right\rangle,
$$
where $b \in \mathbb{R}^p$, $a\in \mathbb{S}$, and $\lambda \in \mathbb{R} \backslash \left\{ 0 \right\} $. This group gives rise to M\"obius transformations of the form
\begin{align*}
	\varphi(\boldsymbol{x}) := (a \boldsymbol{x} + b)(c \boldsymbol{x} + d)^{-1}, \quad \text{with } M = \begin{pmatrix} a & b \\ c & d \end{pmatrix} \in \mathrm{GRAV}(\mathbb{R}^p \oplus \mathbb{R}^q),
\end{align*}
accompanied by the conformal weight
\begin{align*}
	J_k(M,\boldsymbol{x}) := \frac{\overline{c \boldsymbol{x} + d}}{|c \boldsymbol{x} + d|^{p-k+2}},\ J_{-k}(M,\boldsymbol{x}) := \frac{{c \boldsymbol{x} + d}}{|c \boldsymbol{x} + d|^{p+k+2}}.
\end{align*}
Here, $\boldsymbol{x} = \boldsymbol{x}_p + \underline{\omega} r \in \Omega$ denotes a $(p,q)$-type variable with $\boldsymbol{x}_p \in \mathbb{R}^p$, $\underline{\omega} \in \mathbb{S}^{q-1} \subset \mathbb{R}^q$, and $r > 0$.

 \begin{theorem}[Conformal covariance of $T_k$]
 	Let $g \in C^1(\Omega _D)$, and let $M \in \mathrm{GRAV}(\mathbb{R}^p \oplus \mathbb{R}^q)$. Then the Teodorescu transform satisfies the conformal covariance identity
 	\begin{align*}
 	(T_k g)(\varphi(\boldsymbol{x})) = J^{-1}_k(M, \boldsymbol{x}) \cdot T_k\left[ J_{-k}(M,\cdot) \cdot g \circ \varphi \right](\boldsymbol{x}).
 	\end{align*}
 \end{theorem}
 
 \begin{proof}
 	We begin by evaluating
 	\begin{align*}
 		(T_k g)(\varphi(\boldsymbol{x})) = \int_{\varphi(\Omega_D)} \mathcal{E}_{k, \boldsymbol{y}}(\varphi(\boldsymbol{x})) \cdot g(\boldsymbol{y}) \, dV(\boldsymbol{y}).
 	\end{align*}
 	Change variables by setting $ \boldsymbol{y} = \varphi(\boldsymbol{z}) $, so that $ \boldsymbol{z} = \varphi^{-1}(\boldsymbol{y}) $, and write
 	\begin{align*}
 		dV(\boldsymbol{y}) = |\det D\varphi(\boldsymbol{z})| \, dV(\boldsymbol{z}).
 	\end{align*}
 	Thus,
 	\begin{align*}
 		(T_k g)(\varphi(\boldsymbol{x})) = \int_{\Omega_D} \mathcal{E}_{k, \varphi(\boldsymbol{z})}(\varphi(\boldsymbol{x})) \cdot g(\varphi(\boldsymbol{z})) \cdot |\det D\varphi(\boldsymbol{z})| \, dV(\boldsymbol{z}).
 	\end{align*}
 	
 	According to the known transformation property of the Cauchy kernel in Clifford analysis in \cite[Theorem 6]{Ryan}, the poly-partial-slice Cauchy kernel satisfies
 	\begin{align*}
 		\mathcal{E}_{k, \varphi(\boldsymbol{z})}(\varphi(\boldsymbol{x})) = J_k^{-1}(M, \boldsymbol{x}) \cdot \mathcal{E}_{k, \boldsymbol{z}}(\boldsymbol{x}) \cdot J^{-1}_k(M, \boldsymbol{z}).
 	\end{align*}
 Substituting into the integral
 \begin{align*}
 	&(T_k g)(\varphi(\boldsymbol{x})) \\
 	=& \int_{\Omega_D} \mathcal{E}_{k, \varphi(\boldsymbol{z})}(\varphi(\boldsymbol{x})) \, g(\varphi(\boldsymbol{z})) \, |\det D\varphi(\boldsymbol{z})| \, dV(\boldsymbol{z}) \\
 	=& \int_{\Omega_D} J^{-1}_k(M, \boldsymbol{x}) \cdot \mathcal{E}_{k, \boldsymbol{z}}(\boldsymbol{x}) \cdot J^{-1}_k(M, \boldsymbol{z}) \cdot g(\varphi(\boldsymbol{z})) \, |\det D\varphi(\boldsymbol{z})| \, dV(\boldsymbol{z}) \\
 	=& J^{-1}_k(M, \boldsymbol{x}) \cdot \int_{\Omega_D} \mathcal{E}_{k, \boldsymbol{z}}(\boldsymbol{x}) \cdot J_{-k}(M, \boldsymbol{z}) \cdot g(\varphi(\boldsymbol{z}))  dV(\boldsymbol{z}),
 \end{align*}
where $|\det D\varphi(\boldsymbol{z})|=J_k(M, \boldsymbol{z})J_{-k}(M, \boldsymbol{z}) $ is the Jacobian, we conclude that
 \begin{align*}
 	(T_k g)(\varphi(\boldsymbol{x})) = J^{-1}_k(M, \boldsymbol{x}) \cdot T_k\left[ J_{-k}(M, \cdot) \cdot g \circ \varphi \right](\boldsymbol{x}),
 \end{align*}
 this completes the proof.
 \end{proof}

 Then we study how $T_k$ and the differential operator $\overline{\vartheta}$ interact. 
 The next theorem shows that they can be interchanged inside the integral.
 
\begin{theorem}
	Let $1 \leq s < \infty$, and let $f \in L^s(\Omega_D, \mathbb{R}_{p+q})$. Then for almost every $\boldsymbol{x} \in \Omega_D$, we have
	\begin{align*}
		\overline{\vartheta} \left( T_k f \right)(\boldsymbol{x}) = \int_{\Omega_D} \left( \overline{\vartheta} \mathcal{E}_{k,\boldsymbol{y}}(\boldsymbol{x}) \right) f(\boldsymbol{y}) \, dV(\boldsymbol{y}).
	\end{align*}
\end{theorem}

\begin{proof}
	We prove the statement in the case where \(k\) is odd; the even case can be treated in a similar way.
	
	To justify the interchange of the differential operator $\overline{\vartheta}$ and the integral, we apply H\"older's inequality with $1/s + 1/s' = 1$
	\begin{align*}
		|\overline{\vartheta} T_k f(\boldsymbol{x})| \leq \left( \int_{\Omega_D} |\overline{\vartheta} \mathcal{E}_{k,\boldsymbol{y}}(\boldsymbol{x})|^{s'} \, dV(\boldsymbol{y}) \right)^{1/s'} \cdot \|f\|_{L^s(\Omega_D)}.
	\end{align*}
	From the explicit form of $\mathcal{E}_k$ for odd $k$, we have the pointwise estimate
	\begin{align*}
		\left| \overline{\vartheta} \mathcal{E}_{k,\boldsymbol{y}}(\boldsymbol{x}) \right| \leq C \left( \frac{1}{|\boldsymbol{y} - \boldsymbol{x}_{\underline{\omega}}|^{p - k + 1}} + \frac{1}{|\boldsymbol{y} - \boldsymbol{x}_{-\underline{\omega}}|^{p - k + 1}} \right).
	\end{align*}
	To verify the integrability near the singularity, we perform a change to spherical coordinates. Set $ \boldsymbol{y} = \boldsymbol{x}_{\underline{\omega}} + s\boldsymbol{t} $, where $\boldsymbol{t} \in \mathbb{S}^{p+q-1}$, then
	\begin{align*}
		\int_{B_\varepsilon(\boldsymbol{x}_{\underline{\omega}})} \frac{1}{|\boldsymbol{y} - \boldsymbol{x}_{\underline{\omega}}|^{p-k+1}} \, dV(\boldsymbol{y}) &= \int_0^\varepsilon \int_{\mathbb{S}^{p+q-1}} \frac{1}{s^{p-k+1}} \cdot s^{p+q-1} \, d\sigma(\boldsymbol{t}) \, ds \\
		&= C \int_0^\varepsilon s^{k+q - 2} \, ds,
	\end{align*}
	which converges near zero provided $k + q - 2 > -1$, i.e., $k \geq \max(1, 2 - q)$. Since we assumed $k \geq 1$, the integral is finite.
	
	Furthermore, since $f \in L^s(\Omega_D)$ and has compact support, and $\overline{\vartheta} \mathcal{E}_{k,\boldsymbol{y}}(\boldsymbol{x})$ is locally integrable in $\boldsymbol{y}$ for fixed $\boldsymbol{x}$, the dominated convergence theorem applies and yields the result
	\begin{align*}
		\overline{\vartheta} T_k f(\boldsymbol{x}) = \int_{\Omega_D} \left( \overline{\vartheta} \mathcal{E}_{k,\boldsymbol{y}}(\boldsymbol{x}) \right) f(\boldsymbol{y}) \, dV(\boldsymbol{y}).
	\end{align*}
\end{proof}

The previous theorem shows that the differential operator $\overline{\vartheta}$ can be interchanged with the integral in $T_k$.
A natural question is whether this commutation still holds when applying $\overline{\vartheta}$ multiple times.
The following corollary confirms this by extending the result to $\overline{\vartheta}^m$. 
\begin{corollary}\label{4.6}
	Let $ m \leq k $, and $1 \leq s < \infty $. Then for any $ f \in L^s(\Omega_D, \mathbb{R}_{p+q}) $, we have
	\begin{align*}
		\overline{\vartheta}^m \left( T_k f \right)(\boldsymbol{x}) = \int_{\Omega_D} \left( \overline{\vartheta}^m \mathcal{E}_{k,\boldsymbol{y}}(\boldsymbol{x}) \right) f(\boldsymbol{y}) \, dV(\boldsymbol{y})
	\end{align*}
	for almost every $ \boldsymbol{x} \in \Omega_D $, provided the singularity order of $ \overline{\vartheta}^m \mathcal{E}_k $ is less than $ p $.
	In particular, when $ m = k $, we obtain $ \overline{\vartheta}^k T_k f = f $.
\end{corollary}

\begin{proof}
	We proceed by mathematical induction on $m$. The case $m = 1$ was established in the previous theorem.
	
	We only consider the case when $k$ is odd. Suppose the statement holds for $m$. That is,
	\begin{align*}
		\overline{\vartheta}^m \left( T_k f \right)(\boldsymbol{x}) = \int_{\Omega_D} \left( \overline{\vartheta}^m \mathcal{E}_{k,\boldsymbol{y}}(\boldsymbol{x}) \right) f(\boldsymbol{y}) \, dV(\boldsymbol{y}).
	\end{align*}
	Now consider $\overline{\vartheta}^{m+1}$. We apply $\overline{\vartheta}$ to both sides and obtain
	\begin{align*}
		\overline{\vartheta}^{m+1} (T_k f)(\boldsymbol{x}) = \overline{\vartheta} \int_{\Omega_D} \overline{\vartheta}^m \mathcal{E}_{k,\boldsymbol{y}}(\boldsymbol{x}) f(\boldsymbol{y}) \, dV(\boldsymbol{y}).
	\end{align*}
	We assume the following pointwise estimate holds
	\begin{align*}
		\left| \overline{\vartheta}^{m+1} \mathcal{E}_{k,\boldsymbol{y}}(\boldsymbol{x}) \right| \leq C \left( \frac{1}{|\boldsymbol{y} - \boldsymbol{x}_{\underline{\omega}}|^{p - k + m + 1}} + \frac{1}{|\boldsymbol{y} - \boldsymbol{x}_{-\underline{\omega}}|^{p - k + m + 1}} \right).
	\end{align*}
	As before, we verify local integrability near $\boldsymbol{x} = \boldsymbol{y}$ via spherical coordinates. Set $\boldsymbol{y} = \boldsymbol{x}_{\underline{\omega}} + s\boldsymbol{t}$, and compute
	\begin{align*}
		\int_{B_\varepsilon(\boldsymbol{x}_{\underline{\omega}})} \frac{1}{|\boldsymbol{y} - \boldsymbol{x}_{\underline{\omega}}|^{p - k + m + 1}} \, dV(\boldsymbol{y}) &= \int_0^\varepsilon \int_{\mathbb{S}^{p+q-1}} \frac{1}{s^{p - k + m + 1}} \cdot s^{p+q-1} \, d\sigma(\boldsymbol{t}) \, ds \\
		&= C \int_0^\varepsilon s^{k + q - m - 2} \, ds,
	\end{align*}
	which is finite near 0 provided $k + q - m - 2 > -1$, i.e., $m < k + q -1$. Thus the integral is absolutely convergent.
	
	Therefore,
	\begin{align*}
		\overline{\vartheta}^{m+1} (T_k f)(\boldsymbol{x}) = \int_{\Omega_D} \overline{\vartheta}^{m+1} \mathcal{E}_{k,\boldsymbol{y}}(\boldsymbol{x}) f(\boldsymbol{y}) \, dV(\boldsymbol{y}).
	\end{align*}
	This completes the proof.
\end{proof}
\begin{remark}
	As a consequence of Corollary~\ref{4.6} and the conformal covariance of $T_k$, since $\overline{\vartheta}^k$ is the left inverse of $T_k$ as stated in Corollary \ref{4.6}, we can observe that the differential operator $\overline{\vartheta}^k$ satisfies the following conformal covariance identity.
	\begin{align*}
		\left( \overline{\vartheta}^k f \right)(\varphi(\boldsymbol{x})) 
		= J_{-k}^{-1}(M, \boldsymbol{x}) \cdot \overline{\vartheta}^k J_k(M, \bx) \cdot \left[ f \circ \varphi \right](\boldsymbol{x}),
	\end{align*}
	which shows that $ \overline{\vartheta}^k $ transforms covariantly under $\mathrm{GRAV}(\mathbb{R}^p \oplus \mathbb{R}^q) $, and this coincides with the result obtained in \cite[Theorem 3]{DX}.
\end{remark}
\subsection*{Acknowledgments}
The work of Chao Ding is supported by National Natural Science Foundation of China (No. 12271001,12571078). 
 \subsection*{Data Availability}
No new data were created or analysed during this study. Data sharing is not applicable to this article.


\end{document}